\documentclass[12pt]{amsart}

\usepackage[margin=1.25in]{geometry} 

\usepackage{xcolor} 
\usepackage{amsthm,amssymb}
\allowdisplaybreaks
\usepackage{enumerate,enumitem,nicefrac,listings,longtable,caption,makecell,textcomp}
\usepackage[pdfencoding=auto,psdextra,colorlinks,linkcolor=blue,citecolor=blue]{hyperref}
\usepackage[pagewise]{lineno}\nolinenumbers 

\newtheorem{theorem}{Theorem}[section]
\theoremstyle{plain}
\newtheorem{lemma}[theorem]{Lemma}
\newtheorem{proposition}[theorem]{Proposition}

\theoremstyle{definition}

\newcommand{\q}[1]{\overline{#1}}
\newcommand{\mc}[1]{\mathcal{#1}}
\newcommand{\mbf}[1]{\mathbf{#1}}
\newcommand{\mr}[1]{\mathrm{#1}}
\newcommand{\gen}[1]{\langle {#1} \rangle}
\newcommand{\nrm}{\trianglelefteq}
\newcommand{\wt}[1]{\widetilde{#1}}
\newcommand{\Z}{\mathbb{Z}} 
\newcommand{\Cen}[2]{\mathbf{C}_{#1}(#2)} 
\newcommand{\Norm}[2]{\mathbf{N}_{#1}(#2)} 

\renewcommand{\sl}[2]{\mathrm{SL}_{#1}(#2)} 
\newcommand{\pgl}[2]{\mathrm{PGL}_{#1}(#2)} 
\newcommand{\psl}[2]{\mathrm{PSL}_{#1}(#2)} 
\newcommand{\su}[2]{\mathrm{SU}_{#1}(#2)}   
\newcommand{\psu}[2]{\mathrm{PSU}_{#1}(#2)} 
\renewcommand{\sp}[2]{\mathrm{Sp}_{#1}(#2)} 
\newcommand{\psp}[2]{\mathrm{PSp}_{#1}(#2)} 

\definecolor{codegray}{rgb}{0.5,0.5,0.5}
\definecolor{codepurple}{rgb}{0.58,0,0.82}
\definecolor{backcolour}{rgb}{0.95,0.95,0.92}

\lstdefinestyle{mystyle}{
    backgroundcolor=\color{backcolour},   
    commentstyle=\color{black},
    keywordstyle=\color{black},
    numberstyle=\tiny\color{codegray},
    stringstyle=\color{codepurple},
    basicstyle=\ttfamily\footnotesize,
    breakatwhitespace=false,         
    breaklines=true,                 
    captionpos=b,                    
    keepspaces=true,                 
    numbers=none,
    numbersep=5pt,                  
    showspaces=false,                
    showstringspaces=false,
    showtabs=false,                  
    tabsize=2
}

\title[Maximal subgroups of almost odd index]{Finite groups whose maximal subgroups have almost odd index}
\author{Christopher A. Schroeder and Hung P. Tong-Viet}
\date{\today}

\begin{document}

\maketitle

\vspace{-7pt}
\begin{quote}
\footnotesize
\textsc{Abstract.} A recurring theme in finite group theory is understanding how the structure of a finite group is determined by the arithmetic properties of group invariants. There are results in the literature determining the structure of finite groups whose irreducible character degrees, conjugacy class sizes or indices of maximal subgroups are odd. These results have been extended to include those finite groups whose character degrees or conjugacy class sizes are not divisible by $4$. In this paper, we determine the structure of finite groups whose maximal subgroups have index not divisible by $4$. As a consequence, we obtain some new $2$-nilpotency criteria.
\end{quote}
\vspace{15pt}

\section{Introduction}\label{sec:intro}

There is a long history of deducing statements about the structure of finite groups from the arithmetic structure of group invariants. Character degrees, conjugacy class sizes and indices of maximal subgroups are all invariants that have been considered in the literature. The prime $2$ often plays a prominent role in these investigations. To simplify our exposition, we sometimes say that an integer not divisible by $4$ is ``almost odd''. Note that an odd integer is also almost odd.

With regard to character degrees, the It{\^o}--Michler theorem says that the irreducible character degrees of a finite group~$G$ are odd if and only if $G$ has a normal and abelian Sylow $2$-subgroup \cite{Na18}. In this case, $G$ is solvable by the Feit--Thompson theorem. Lewis relaxed this condition in~\cite{Le07}, where he showed that the irreducible character degrees of a nonsolvable finite group $G$ are almost odd if and only if $G \cong \mr{A}_7 \times S$, where $S$ is solvable with odd irreducible character degrees. Turning to conjugacy classes, Chillag and Herzog proved in~\cite{CH90} that all conjugacy classes of a finite group $G$ have odd size if and only if a Sylow $2$-subgroup of $G$ is central, so $G$ is solvable. In the same paper, they show that if the size of every conjugacy class of a finite group $G$ is almost odd, then $G$ is solvable. These results on the sizes of conjugacy classes still hold upon restricting to real conjugacy class sizes. Dolfi, Navarro and Tiep proved in \cite[Theorem 6.1]{DNT08} that a finite group is solvable if its real class sizes are odd, and it is an immediate consequence of a result of the second-named author that a finite group is solvable if its real class sizes are almost odd~\cite{HTV13}. For maximal subgroups, Guralnick proved in~\cite{Gu86} that if a finite group $G$ has even order and its maximal subgroups have odd index, then $G/\mbf{O}_{2',2}(G) \cong \mr{A}_7$. In this paper, we weaken this hypothesis and determine the structure of finite groups whose maximal subgroups have almost odd index. All groups in this paper are finite.

In our first theorem, we solve this problem for almost simple groups using the classification of finite simple groups.

\begin{theorem}\label{thm1:almost-simple} 
Let $G$ be a finite almost simple group with simple socle $S$. Every maximal subgroup of $G$ has index not divisible by $4$ if and only if $G$ is isomorphic to $\mr{A}_6$, $\mr{S}_6$, $\mr{A}_7$, $\psu{3}{5}$, $\psu{3}{5}.2$ or $S$ belongs to one of the two infinite families $\psl{2}{q}$ with $q \equiv 5 \, (8)$ and $\mr{G}_2(q)$ with $q=5^{2n+1}$ $(n\ge 0$, $n \neq 1)$.
\end{theorem}

In fact, we will prove a statement that is slightly stronger than Theorem \ref{thm1:almost-simple} in Theorem~\ref{thm:almost-simple-faithful}, and this stronger statement will be required in our proof of Theorem~\ref{thm3:almost-odd}. Turning to solvable groups, our next result is a recognition theorem for $p$-nilpotency that we have not found in the literature. Note that if $G$ is $p$-solvable for the smallest prime $p$ dividing $|G|$, then $G$ is solvable by the Feit--Thompson theorem.

\begin{theorem}\label{thm2:solvable}
Let $G$ be a finite $p$-solvable group, where $p$ is the smallest prime dividing the order of $G$. Every maximal subgroup of $G$ has index not divisible by $p^2$ if and only if $G$ has a normal $p$-complement.
\end{theorem}

Recall that $\mbf{O}_{2'}(G)$ is the largest normal subgroup of odd order in a group $G$, and $\mbf{O}_2(G)$ is the largest normal $2$-subgroup of~$G$. Define the normal subgroup $\mbf{O}_{2',2}(G)$ of $G$ by $\mbf{O}_{2',2}(G)/\mbf{O}_{2'}(G)=\mbf{O}_2(G/\mbf{O}_{2'}(G))$. Note that Theorem \ref{thm2:solvable} says that if $G$ is a finite solvable group, then its maximal subgroups have almost odd index if and only if $G=\mbf{O}_{2',2}(G)$. In our next theorem, we describe the structure of nonsolvable finite groups in which all maximal subgroups have almost odd index.

\begin{theorem}\label{thm3:almost-odd}
Let $G$ be a nonsolvable finite group, and let $\q{G}=G/\mbf{O}_{2',2}(G)$. If the index of every maximal subgroup of $G$ is not divisible by $4$, then \mbox{$\q{G}' \cong N_1 \times \dots \times N_r$}, where each $N_i$ is a simple minimal normal subgroup of $\q{G}$ and is isomorphic to one of the simple groups
\begin{enumerate}[nolistsep,label=\textup{(\arabic*)}]
    \item $\mr{A}_6$, $\mr{A}_7$, $\psu{3}{5}$,
    \item $\psl{2}{q}$ with $q \equiv 5  \, (8)$, or
    \item $\mr{G}_2(q)$ with $q=5^{2n+1}$ $(n \ge 0$, $n \neq 1)$.
\end{enumerate}
If $N_i \cong \mr{A}_6$, then $\mr{A}_6 \le \q{G}/\Cen{\q{G}}{N_i} \le \mr{S}_6$; if $N_i \cong \mr{A}_7$, then $\q{G}/\Cen{\q{G}}{N_i} \cong \mr{A}_7$; if $N_i \cong \psu{3}{5}$, then $\psu{3}{5} \le \q{G}/\Cen{\q{G}}{N_i} \le \psu{3}{5}.2$. The Sylow $2$-subgroups of $\q{G}$ have derived length at most $2$.
\end{theorem}

The statement of Theorem \ref{thm3:almost-odd} is in some sense the best possible: As the example $\mr{S}_6 \times \mr{S}_6$ shows, $\q{G}'$ can be a direct product of minimal normal simple groups, and the minimal normal subgroups need not be distinct. However, if two minimal normal subgroups are isomorphic, then there must be an extension by automorphisms, as a diagonal subgroup of a product of two isomorphic nonabelian simple groups is maximal with index divisible by $4$.  Furthermore, $\q{G}$ need not split as a direct product; there exists a counterexample of the form $(\mr{A}_5 \times \mr{A}_6).2$, where the outer automorphism acts ``diagonally'' on the socle.

Theorems \ref{thm2:solvable} and \ref{thm3:almost-odd} yield a $2$-nilpotency criterion for general finite groups. 

\begin{theorem}\label{cor4:2nilpotent-1}
Let $G$ be a finite group. If every maximal subgroup of $G$ has index not divisible by $3$ and not divisible by $4$, then $G$ has a normal $2$-complement.
\end{theorem}

Recall that a subgroup $H$ of a finite group $G$ is called second maximal (sometimes also weakly second maximal) if $H$ is a maximal subgroup of a maximal subgroup of~$G$. Our last result is a recognition theorem for $p$-nilpotency in terms of the indices of second maximal subgroups.

\begin{theorem}\label{cor5:2nilpotent-2}
Let $G$ be a finite group, and let $p$ be the smallest prime dividing the order of $G$. Then $G$ has a normal $p$-complement if and only if for every second maximal subgroup $H$ of $G$, the index $[G:H]$ is a power of $p$ or is not divisible by $p^2$.
\end{theorem}

Note that one direction of Theorem~\ref{cor5:2nilpotent-2} does not hold if $p$ is not the smallest prime that divides $|G|$. For example, the order of $\mr{A}_5$ is not divisible by $3^2$, so no second maximal subgroup has index divisible by $3^2$. But $\mr{A}_5$ is simple, so it does not contain a normal $3$-complement. The other direction does not require $p$ to be the smallest prime dividing~$|G|$, as we will see.

Our paper is organized as follows. In Section \ref{sec:almost-simple}, we consider almost simple groups whose maximal subgroups have almost odd index, and we prove Theorem \ref{thm1:almost-simple}. In Section \ref{sec:finite-groups}, we consider general finite groups and prove Theorem \ref{thm3:almost-odd}. Finally, we turn to $p$-nilpotency criteria in Section \ref{sec:solvable}, proving Theorems \ref{thm2:solvable}, \ref{cor4:2nilpotent-1} and~\ref{cor5:2nilpotent-2}. For the convenience of the reader, we have included the GAP \cite{GAP4} code for our simple computer calculations in Section \ref{sec:gap}.

\section{Almost simple groups}\label{sec:almost-simple}
In this section, we classify those almost simple groups whose maximal subgroups have almost odd index. A maximal subgroup of an almost simple group that does not contain the socle is said to be faithful, as such a subgroup is corefree and so the permutation representation on the cosets of the maximal subgroup is faithful. The main theorem of this section shows that an almost simple group has a faithful maximal subgroup of index divisible by $4$ if and only if it is not on the list in Theorem~\ref{thm1:almost-simple}, and then Theorem~\ref{thm1:almost-simple} follows. We let $q=p^f$ for some prime $p$ and some integer $f \ge 1$.

\begin{theorem}\label{thm:almost-simple-faithful} 
Let $G$ be a finite almost simple group with simple socle $S$. Then $G$ has a faithful maximal subgroup of index divisible by $4$ if and only if $S \not\cong \psl{2}{q}$ with $q \equiv 5 \, (8)$, $S \not\cong \mr{G}_2(q)$ with $q=5^{2n+1}$ $(n \ge 0$, $n \neq 1)$ and $G$ is not isomorphic to $\mr{A}_6$, $\mr{S}_6$, $\mr{A}_7$, $\psu{3}{5}$ or $\psu{3}{5}.2$.
\end{theorem}

We prove Theorem \ref{thm:almost-simple-faithful} using the classification of finite simple groups. Before we begin, we prove a lemma describing a situation in which a maximal subgroup of a simple group gives rise to a faithful maximal subgroup of an almost simple group with the same index. This lemma simplifies our proofs, and it will be used repeatedly for the finite simple groups of Lie type.

Let $G$ be an almost simple group with simple socle $S$, so $S \le G \le \mr{Aut}(S)$, and let $A = \mr{Aut}(S)$. Let $H_S$ be a maximal subgroup of $S$. Let $\Delta$ be the set of $A$-conjugates of $H_S$. Since $S \nrm A$, these conjugates lie in $S$ and the $S$-orbits form a system of blocks for the $A$-action. Letting $c$ be the number of blocks, the permutation action of $A$ on these blocks yields a homomorphism $\pi:A \rightarrow \mr{S}_c$ that contains $S$ in the kernel. If $c=1$ (in which case we also say $\pi=1$), this means that $S$ acts transitively on $\Delta$ by conjugation, and so $[S:\Norm{S}{H_S}] = |\Delta| $. Since $H_S$ is a maximal subgroup of the simple group $S$, we have $\Norm{S}{H_S}=H_S$, and so 
\begin{align}\label{eq:c=1}
    [S:H_S]=[G:\Norm{G}{H_S}]
\end{align}
since $S$ acting transitively on $\Delta$ means $G$ certainly does as well. 

Using the notation of the last paragraph, we can now state the lemma.

\begin{lemma}\label{l:c=1}
If $c=1=\pi$, then $H_G=\Norm{G}{H_S}$ is a faithful maximal subgroup of $G$ and we have \mbox{$[G:H_G]=[S:H_S]$}.
\end{lemma}

\begin{proof}
Let $H_G$ be a maximal subgroup of $G$ containing $\Norm{G}{H_S}$. 

We first show that $c=1$ implies that $S \not\le H_G$, so $H_G$ is a faithful subgroup of $G$. Since $S \nrm G$, we have that $S \Norm{G}{H_{S}} \le G$, and in fact
\begin{align*}
    |S \Norm{G}{ H_{S} }| = \frac{|S| |\Norm{ G }{ H_{S} }| }{|\Norm{S}{H_{S}}|} = [S:H_{S}] | \Norm{ G}{ H_{S} } | = |G|,
\end{align*}
where the last equality follows from Equation (\ref{eq:c=1}) above since $c=1$. Thus, the proper subgroup $H_G < G$ cannot contain both $\Norm{ G }{ H_{S} }$ and $S$.

Next, we show that $H_G = \Norm{G}{H_{S}}$. Assume for contradiction that $H_G$ properly contains the normalizer. Then since $S \nrm G$, we have $H_{S} < \gen{H_{S}^{H_G}} \le S$, so we must have equality since $H_{S}$ is maximal in $S$. But then $S = \gen{H_{S}^{H_G}} \le H_G$ since $H_{S} \le H_G$, a contradiction. 

Finally, $[G:H_G]=[S:H_{S}]$ follows immediately from $H_G=\Norm{G}{H_S}$ and Equation~(\ref{eq:c=1}).
\end{proof}

Now we proceed with the proof of Theorem \ref{thm:almost-simple-faithful} in a series of propositions. In short, we will identify a maximal subgroup of index divisible by $4$ in a simple group with $c=1=\pi$ and use Lemma \ref{l:c=1} to conclude for almost simple groups. We begin with the finite simple groups of Lie type. Our reference for the order formulas of the finite groups of Lie type is \cite[Table 24.1 and Table 24.2]{MT11}. 
If $n$ is an integer and $p$ a prime, we denote by $n_p$ the largest power of $p$ that divides~$n$. 

\begin{proposition}\label{p:psln} 
Let $G$ be an almost simple group with socle $\mr{PSL}_n(q)$ and $n \ge 3$. Then $G$ has a faithful maximal subgroup of index divisible by $4$.
\end{proposition}

\begin{proof}
Let $G = \mr{PSL}_n(q)$, so ${|G|=\tfrac{1}{d}q^{n(n-1)/2}\prod_{k=2}^n (q^k-1)}$ with ${d=(q-1,n)}$. First assume that $n \ge 5$. By \cite{BHR13} for $5 \le n \le 12$ and \cite[Table 3.5.A]{KL90} for ${n \ge 13}$, a subgroup $H \le \mr{PSL}_n(q)$ in Aschbacher class $\mc{C}_3$ of type $\mr{GL}_m(q^r)$ ($n=mr$, $m \ge 1$, $r$~prime) is always maximal and satisfies $c=1=\pi$. By \cite[Proposition 4.3.6]{KL90}, such a subgroup $H \in \mc{C}_3$ has order
\begin{align*}
    |H|=abr|\mr{PSL}_m(q^r)|, \quad \text{where} \quad a=\frac{(q-1,m)(q^r-1)}{(q-1)(q-1,n)}, \quad b=\frac{(q^r-1,m)}{(q-1,m)},
\end{align*} and index
\begin{align*}
    [G:H] = \frac{q^{n(n-1)/2}}{q^{rm(m-1)/2} r} \frac{\prod\limits_{k=1}^{n} (q^k-1)}{\prod\limits_{k=1}^{m} (q^{rk}-1)}.
\end{align*}
We show this index is divisible by $4$. First assume that $q$ is a power of $2$. Then $[G:H]_2 = (1/r_2) q^{\frac12 n(n-m)} > q^n/2 \ge 4$ since $n \ge 5$ means $n-m > 2$. Now assume that $q$ is odd. 
Every $r$-th factor of $q^k-1$ in the numerator cancels, and so we have
\begin{align*}
    [G:H]_2 = \frac{1}{r_2} \prod\limits_{j=0}^{m-1} \prod\limits_{k=jr+1}^{(j+1)r-1} (q^k - 1)_2 \ge \frac{1}{r_2} (2^{r-1})^m = \frac{1}{r_2} 2^{n-m} \ge 4
\end{align*}
since $n \ge 5$ means $n-m>2$. We have shown that $[G:H]$ is always divisible by $4$. Now, since $c=\pi=1$, Lemma~\ref{l:c=1} shows that such a maximal subgroup $H \le \mr{PSL}_n(q)$ gives rise to a faithful maximal subgroup of any almost simple group with socle $\mr{PSL}_n(q)$ with index divisible by $4$.

To finish the proof of the proposition, it remains to consider $n=3$ and ${n=4}$. 

If ${n=4}$, then $H \in \mc{C}_3$ 
has index ${[G:H]=\frac12 q^4(q-1)(q^3-1)}$ by our calculation above and satisfies $c=1=\pi$ by \cite[Table 8.8]{BHR13}. This index is divisible by $4$ for $q$ even and for $q \equiv 1 \, (4)$. If $q \equiv 3 \, (4)$, then $H \in \mc{C}_8$ of type $\mr{SO}^-_4(q).2$ has index ${[G:H]=\frac12 q^4(q^2-1)(q^3-1)}$ by \cite[Table~8.8]{BHR13}, which is divisible by $4$. Since $q \equiv 3\, (4)$, we have ${d=(q-1,4)=2}$, and so $c=d/2=1$ by \cite[Table 8.8]{BHR13}. This case is finished by Lemma \ref{l:c=1}.

If $n=3$ and $q \neq 4$, then $H \in \mc{C}_3$ has index $[G:H] = \frac13 q^3(q-1)(q^2-1)$ by our calculation above and satisfies $c=1=\pi$ by \cite[Table 8.3]{BHR13}. This index is divisible by~$4$. Finally, if $n=3$ and $q = 4$, then we can find faithful maximal subgroups of index divisible by $4$ by computing with all possible almost simple groups in GAP.
\end{proof}

\begin{proposition}\label{p:psu} 
Let $G$ be an almost simple group with socle $\mr{PSU}_n(q)$ and $n \ge 3$. Then $G$ has a faithful maximal subgroup of index divisible by $4$ unless $G$ is isomorphic to $\psu{3}{5}$ or $\psu{3}{5}.2$.
\end{proposition}

\begin{proof}
Let $G$ be almost simple with simple socle $\psu{n}{q}$ and $n \ge 3$. In anticipation of small exceptions to our arguments, we first do some computations. If $(n,q)$ is one of $(3,3)$, $(3,5)$ or $(5,2)$, we compute in GAP to find faithful maximal subgroups of index divisible by $4$ unless $G \cong \psu{3}{5}$ or $G \cong \psu{3}{5}.2$. If $(n,q)=(6,2)$, then $H \in \mc{C}_1$ of type $\mr{GU}_5(2)$ is a maximal subgroup of $\psu{6}{2}$ of index divisible by $2^5$ and satisfies $c=1=\pi$ by \cite[Table 8.26]{BHR13}. Then we are done in this case by Lemma~\ref{l:c=1}.

Now let $G=\psu{n}{q}$, so $|G| = \tfrac{1}{d} q^{n(n-1)/2} \prod_{k=2}^n (q^k-(-1)^k)$ with $d=(q+1,n)$. By \cite{BHR13} for $3 \le n \le 12$ and \cite[Table 3.5.B]{KL90} for $n \ge 13$, a subgroup $H \le \psu{n}{q}$ in Aschbacher class $\mc{C}_3$ of type $\mr{GU}_m(q^r)$ ($n=mr$, $m \ge 1$, $r \ge 3$~prime) is always maximal (when it exists) with $c=1=\pi$ if $(n,q)$ is not one of $(3,3)$, $(3,5)$, $(5,2)$ or $(6,2)$. We addressed this list of exceptions in the last paragraph. Note that the class $\mc{C}_3$ exists if and only if $n$ is not a power of~$2$. So by \cite[Proposition 4.3.6]{KL90}, if $n$ is not one of these exceptions and $n$ is not a power of $2$, then such a subgroup $H$ has order
\begin{align*}
    |H|=abr|\mr{PSU}_m(q^r)|, \quad \text{where} \quad a=\frac{(q+1,m)(q^r+1)}{(q+1)(q+1,n)}, \quad b=\frac{(q^r+1,m)}{(q+1,m)},
\end{align*} and index
\begin{align*}
    [G:H] = \frac{q^{n(n-1)/2}}{q^{rm(m-1)/2} r} \frac{\prod\limits_{k=1}^{n} (q^k-(-1)^k)}{\prod\limits_{k=1}^{m} (q^{rk}-(-1)^k)}.
\end{align*}
We show that this index is divisible by $4$. First assume that $q$ is a power of $2$. As $r$ is odd, $[G:H]_2 = q^{\frac12 n(n-m)} \ge q^n \ge 4$ since $n \ge 3$ means $n-m \ge 2$. Now assume that $q$ is odd. Since $r \ge 3$ is odd, $q^{rk}-(-1)^k = q^{rk}-(-1)^{rk}$, and so every $r$-th factor of $q^k-(-1)^k$ in the numerator cancels. We have
\begin{align*}
    [G:H]_2 = \prod\limits_{j=0}^{m-1} \prod\limits_{k=jr+1}^{(j+1)r-1} (q^k - (-1)^k)_2 \ge (q+1)_2(q^2-1)_2  \ge 4.
\end{align*}
We have shown that $[G:H]$ is always divisible by $4$. Now, since $c=1=\pi$, Lemma~\ref{l:c=1} shows that such a maximal subgroup $H \le \mr{PSU}_n(q)$ gives rise to a faithful maximal subgroup of any almost simple group with socle $\mr{PSU}_n(q)$ with index divisible by $4$.

Now assume that $G$ is almost simple with socle $\psu{n}{q}$ and $n \ge 3$ is a power of~$2$. If $n=4$, then a subgroup $H \in \mc{C}_1$ of type $\mr{GU}_3(q)$ has index $q^3(q-1)(q^2+1)$ in $\psu{4}{q}$ and satisfies $c=1=\pi$ by \cite[Table 8.10]{BHR13}. This index is divisible by $4$. So we may assume that $n \ge 8$. By \cite[Table 8.46]{BHR13} for $n=8$ and \cite[Table 3.5.B]{KL90} for $n \ge 13$, a subgroup $H \le \psu{n}{q}$ in Aschbacher class $\mc{C}_2$ of type $\mr{GL}_{n/2}(q^2).2$ is maximal and satisfies $c=1=\pi$. Therefore, it suffices by Lemma \ref{l:c=1} to assume that ${G \cong \psu{n}{q}}$ and show that $4$ divides $[G:H]$. By \cite[Proposition 4.2.4]{KL90}, we have $|H|=2a |\psl{n/2}{q^2}|$, where $a=(q-1)(q^2-1,\frac{n}{2})/(q+1,n)$. Therefore, the index is
\begin{align*}
    [G:H] = \frac{q^{n(n-1)/2}}{2 q^{n(n-2)/4}} \frac{\prod\limits_{k=2}^n (q^k-(-1)^k)}{(q-1)\prod\limits_{k=2}^{n/2} (q^{2k}-1)} = \frac12 q^{n^2/4} \prod\limits_{\substack{k=1 \\ \text{odd}}}^n (q^k+1).
\end{align*}
This index is clearly divisible by $4$ if $q$ is even, and if $q$ is odd then $[G:H]_2 \ge \big( \frac12 (q+1)(q^3+1)(q^5+1) \big)_2 \ge 4$. We are done.
\end{proof}

In the proof of the last propositions for $\mr{PSL}_n(q)$ and $\psu{n}{q}$, it was important that $n \ge 3$. Our next two propositions address the case $\mr{PSL}_2(q) \cong \mr{PSU}_2(q)$.

\begin{proposition}\label{p:psl2} 
Let $G$ be an almost simple group with socle $\mr{PSL}_2(q)$. If $q \neq 9$ and $q \not\equiv 5 \, (8)$, then $G$ has a faithful maximal subgroup of index divisible by $4$.
\end{proposition}

\begin{proof}
We refer to \cite[Table 8.1]{BHR13} for the maximal subgroups of $\psl{2}{q}$. First assume that $q$ is even. We may assume that $q \ge 8$ since $\psl{2}{2}$ is not simple and ${\psl{2}{4} \cong \psl{2}{5}}$. The maximal subgroups of $\psl{2}{q}$ in class $\mc{C}_3$ isomorphic to $D_{2(q+1)}$ have index ${(q/2)(q-1)}$, which is divisible by $4$. Now assume that $q \equiv 3,7 \, (8)$. The maximal subgroups of $\psl{2}{q}$ in class $\mc{C}_1$ have order $q(q-1)/2$ and index $q+1$, which is divisible by $4$. Finally, assume that $q \equiv 1 \, (8)$. As long as $q \neq 9$, $\psl{2}{q}$ has a maximal subgroup in class $\mc{C}_3$ isomorphic to $D_{q+1}$ of index $q(q-1)/2$, which is divisible by $4$. Since $c=1$ for all the maximal subgroups we have chosen, we have proved the claim for almost simple groups $G$ by Lemma \ref{l:c=1}.
\end{proof}

To finish the case where the almost simple group $G$ has socle $\psl{2}{q}$, it remains to consider $q=9$ and $q \equiv 5 \, (8)$.

\begin{proposition}\label{p:psl2-q5} 
Let $G$ be an almost simple group with socle $\psl{2}{q}$ and $q \equiv 5 \, (8)$ or $q=9$. Then every maximal subgroup of $G$ has almost odd index unless $G$ is isomorphic to $\mr{PGL}_2(9)$, $\mr{M}_{10}$ or $\mr{Aut}(\psl{2}{9})$.
\end{proposition}

\begin{proof}
Let $\psl{2}{9} \le G \le \mr{Aut}(\psl{2}{9})$. We compute in GAP that no maximal subgroup of $G$ has index divisible by $4$ if $G$ is isomorphic to $\mr{PSL}_2(9) \cong \mr{A}_6$ or $\mr{S}_6$, and $G$ contains faithful maximal subgroups of index divisible by $4$ if $G$ is isomorphic to $\mr{PGL}_2(9)$, $\mr{M}_{10}$ or $\mr{Aut}(\psl{2}{9})$.

Now, let $S \cong \psl{2}{q} \le G \le \mr{P \Gamma L}_2(q)$ with $q=p^f$ and $q \equiv 5 \, (8)$. First assume that $G = S$, so $|G|=\frac12 q(q^2-1)$ and $|G|_2=4$. Then a maximal subgroup has index divisible by $4$ if and only if it has odd order. By examining the list of maximal subgroups in \cite[Table 8.1 and 8.2]{BHR13} or \cite[Theorem 2.2]{Gi07}, it is easy to see that this never occurs. Finally, assume $S<G$, and let $M$ be a maximal subgroup of $G$. If $S \le M$, then $M/S$ is a maximal subgroup of $G/S \le \mr{Out}(\mr{PSL}_2(q)) \cong \Z_2 \times \Z_f$, which is abelian. So every such maximal subgroup has prime index and is not divisible by~$4$. Now assume that $S \not\le M$, so $G=SM$ and $[G:M]=[S:M\cap S]$. If $M \cap S$ is maximal in $S$, then its index in $S$ is almost odd by what we have already shown. If $M \cap S$ is not maximal in~$S$, then $M$ is a so-called ``novelty'' maximal subgroup of $G$. The pairs of such $(G,M)$ are given in \cite[Theorem 1.1]{Gi07}. The only pair that possibly satisfies our congruence on $q$ is $G \cong \mr{PGL}_2(q)$ with $q=p \equiv \pm 11, 19 \, (40)$ and $M \cong \mr{S}_4$. But $|G|=q(q^2-1)$, so $|G|_2=8$ and $|M|_2=8$, so $[G:M]$ is odd. We are done.
\end{proof}

Now we turn to the symplectic groups.

\begin{proposition} 
Let $G$ be an almost simple group with socle isomorphic to a simple group $\mr{PSp}_{2n}(q)$. Then $G$ has a faithful maximal subgroup of index divisible by $4$. 
\end{proposition}

\begin{proof}
We may assume that $n \ge 2$ since $\mr{PSp}_{2}(q) \cong \psl{2}{q}$, which we have already considered. First assume that we are not in the case that $n=2$ and $q$ is even. Then for $\psp{2n}{q}$ ($n \ge 2$), a subgroup $H$ in Aschbacher class $\mc{C}_3$ of type $\sp{m}{q^r}$ ($2n=mr$, $r$ prime and $m$ even) is always maximal and satisfies ${c=1=\pi}$ by \cite{BHR13} for $2 \le n \le 6$ and \cite[Table 3.5.C]{KL90} for $n \ge 7$. By Lemma \ref{l:c=1}, the proposition is then proved if we assume that ${G=\psp{2n}{q}}$ with ${|G|=\frac1d q^{n^2} \prod_{k=1}^n (q^{2k}-1)}$ and $d=(q-1,2)$ is simple and we show that the index $[G:H]$ is divisible by~$4$. By \cite[Proposition 4.3.10]{KL90}, we have ${|H|=r|\psp{m}{q^r}|=\frac{r}{d'} (q^r)^{(m/2)^2} \prod_{k=1}^{m/2} (q^{2rk}-1)}$, where $d'=(q^r-1,2)$. Note that $d=d'$ for all $q$. When $q$ is odd, we have ${[G:H]_2 \ge (q^2-1)_2/r_2 \ge 4}$. When $q$ is even, $[G:H]_2=q^{n^2-r(m/2)^2}/r_2=q^{n(n-m/2)}/r_2 \ge 4$ unless $q=2$, $n=2$, $m=2$ and $r=2$, but this case is excluded by our assumption.

Finally, assume that $n=2$ and $q$ is even. We addressed $\psp{4}{2} \cong \mr{S}_6$ in Proposition~\ref{p:psl2-q5}, so assume that $q>2$. By \cite[Table~8.14]{BHR13}, a subgroup $H \le \psp{4}{q}$ in class $\mc{C}_5$ of type $\sp{4}{q_0}$ with $q=q_0^r$ ($r$~prime) is maximal with $c=1=\pi$ and index divisible by $q_0^{4(r-1)}>4$. So we are done in this case by Lemma \ref{l:c=1}, and the proposition is proved.
\end{proof}

Next, we consider the orthogonal groups of odd dimension.

\begin{proposition} 
Let $G$ be an almost simple group with socle isomorphic to a simple group $\mr{P\Omega}_{2n+1}(q)$ $(q$ odd$)$. Then $G$ has a faithful maximal subgroup of index divisible by $4$.
\end{proposition}

\begin{proof}
We may assume that $n\ge 3$ since $\mr{P}\Omega_5(q) \cong \psp{4}{q}$ and $\mr{P}\Omega_3(q) \cong \psl{2}{q}$, which we have already considered. For $\mr{P}\Omega_{2n+1}(q)$ ($q$ odd), a parabolic subgroup $H=P_m$ ($1 \le m \le n$) in Aschbacher class $\mc{C}_1$ is always maximal and satisfies $c=1=\pi$ by \cite{BHR13} for $3 \le n \le 5$ and \cite[Table 3.5.D]{KL90} for $n \ge 6$. By Lemma \ref{l:c=1}, the proposition is then proved if we assume that ${G=\mr{P}\Omega_{2n+1}(q)}$ with ${|G|=\frac12 q^{n^2} \prod_{k=1}^n (q^{2k}-1)}$ is simple and we show that the index $[G:H]$ is divisible by~$4$. Letting $H=P_n$, we have by \cite[Proposition 4.1.20]{KL90} that $|H|=\frac12 q^{n(n+1)/2}|\mr{GL}_n(q)|=\frac12 q^{n^2} \prod_{k=1}^n(q^k-1)$. Therefore, $[G:H]=\prod_{k=1}^n (q^k+1)$. Since $q$ is odd and $n \ge 3$, $[G:H]_2 \ge (q+1)_2 (q^2+1)_2 \ge 4$. 
\end{proof}

To finish with the classical groups, we consider the orthogonal groups of even dimension.

\begin{proposition}\label{p:pomegaminus} 
Let $G$ be an almost simple group with socle a simple group $\mr{P\Omega}_{2n}^{-}(q)$. Then $G$ has a faithful maximal subgroup of index divisible by $4$. 
\end{proposition}

\begin{proof}
We may assume that $n \ge 4$ since $\mr{P\Omega}_6^-(q) \cong \psu{4}{q}$ and $\mr{P\Omega}_4^-(q) \cong \psl{2}{q^2}$, which we have already considered. For $\mr{P\Omega}_{2n}^-(q)$ ($n \ge 4$), subgroups in Aschbacher class $\mc{C}_3$ of type $\mr{GU}_n(q)$ for $n$ odd and $\mr{O}^-_{2n/r}(q^r)$ ($r \mid 2n$, $r$ prime, $2n/r \ge 3$) for $n$ even are always maximal and satisfy $c=1=\pi$ by \cite{BHR13} for $4\le n \le 6$ and \cite[Table~3.5.F]{KL90} for $n \ge 7$. By Lemma \ref{l:c=1}, we can assume that ${G=\mr{P}\Omega_{2n}^-(q)}$ is simple with ${|G|=\frac1d q^{n^2-n} (q^n+1) \prod_{k=1}^{n-1} (q^{2k}-1)}$, where $d=(2,q-1)$ if $n$ is even and $d=(4,q+1)$ if $n$ is odd, and show that the indices of these maximal subgroups are divisible by~$4$.

First assume that $n$ is even. Let $r=2$, so that $2n/r=n \ge 4$ is even, and let $H \in \mc{C}_3$ of type $\mr{O}_n^-(q^2)$. We have by \cite[Proposition 4.3.16]{KL90} that $|H|=2|\mr{P\Omega}_n^-(q^2)|=\frac{2}{d'} (q^2)^{(n/2)^2-n/2}((q^2)^{n/2}+1)\prod_{k=1}^{n/2-1}((q^2)^{2k}-1)=\frac{2}{d'}q^{n^2/2-n}(q^n+1) \prod_{k=1}^{n/2-1} (q^{4k}-1)$, where $d'=(2,q^2-1)$ if $n/2$ is even and $d'=(4,q^2+1)$ if $n/2$ is odd. So for the index we have $[G:H]_2 \ge (d'/2d) (q^{n^2/2})_2(q^2-1)_2 \ge 4$ since when $q$ is odd, $d'$ is $2$ or $4$ and $d = 2$ as $n$ is even.

Now assume that $n$ is odd, and let $H \in \mc{C}_3$ of type $\mr{GU}_n(q)$. We have by \cite[Proposition 4.3.18]{KL90} that $|H|=\frac{q+1}{(4,q+1)}|\psu{n}{q}|(q+1,n)={\frac{1}{(4,q+1)}q^{n(n-1)/2}\prod_{k=1}^n (q^k-(-1)^k)}$. Since $q^{2k}-1=(q^k-(-1)^k)(q^k-(-1)^{k+1})$, we have for the index that $[G:H]=q^{n(n-1)/2}{\prod_{k=1}^{n-1}(q^k-(-1)^{k+1})}$. Since $n\ge 4$, this index is clearly divisible by $4$ whether $q$ is even or odd.
\end{proof}

\begin{proposition} 
Let $G$ be an almost simple group with socle a simple group $\mr{P\Omega}_{2n}^{+}(q)$. Then $G$ has a faithful maximal subgroup of index divisible by $4$. 
\end{proposition}

\begin{proof}
As in Proposition \ref{p:pomegaminus}, we may assume that $n \ge 4$. 
In the same way that we have addressed the other classical groups, we will find a maximal subgroup of $\mr{P\Omega}_{2n}^+(q)$ that has index divisible by $4$ and satisfies $c=1=\pi$. Then we will be done by Lemma~\ref{l:c=1}. So assume that $G=\mr{P\Omega}_{2n}^+(q)$ with $|G|=\frac1d q^{n^2-n} (q^n-1) \prod_{k=1}^{n-1}(q^{2k}-1)$, where $d=(2,q-1)^2$ if $n$ is even and $d=(4,q-1)$ if $n$ is odd.

First assume that $q$ is odd. Let $H \in \mc{C}_1$ be a parabolic subgroup $P_m$ of $G$ with $m=2$. The subgroup $H$ is always maximal and satisfies $c=1=\pi$ by \cite{BHR13} for $4 \le n \le 6$ and \cite[Table~3.5.E]{KL90} for $n \ge 7$. We have from \cite[Proposition 4.1.20]{KL90} that 
\begin{align*}
    |H| = \begin{cases}
    2 q^a (q-1) |\psl{2}{q}| |\mr{P\Omega}_{2n-4}^+(q)|, & -1 \in \Omega_{2n}^+(q) \\
    q^a |\mr{GL}_2(q)| |\Omega_{2n-4}^+(q)|, & -1 \not\in \Omega_{2n}^+(q),
    \end{cases}
\end{align*}
where $a=4n-7$. So ${|H| = \frac{1}{d'} q^{n^2-n}(q-1)(q^2-1)(q^{n-2}-1) \prod_{k=1}^{n-3} (q^{2k}-1)}$, where $d'=2$ if $-1 \not\in \Omega_{2n}^+(q)$ and $d'=d$ if $-1 \in \Omega_{2n}^+(q)$ since $n$ and $n-2$ have the same parity. 
Then
\begin{align*}
    [G:H] &= \frac{d'}{d} \frac{(q^n-1)(q^{2(n-1)}-1)(q^{2(n-2)}-1)}{(q-1)(q^2-1)(q^{n-2}-1)}  \\
    &= \frac{d'}{d} \bigg(\frac{(q^n-1)(q^{n-1}-1)}{(q-1)(q^2-1)}\bigg) (q^{n-1}+1) (q^{n-2}+1). 
\end{align*}
Therefore, $[G:H]_2 \ge (d'/d)(q^{n-1}+1)_2 (q^{n-2}+1)_2$. If $-1 \in \Omega_{2n}^+(q)$, then $d=d'$ and ${[G:H]_2 \ge 4}$. If $-1 \not\in \Omega_{2n}^+(q)$, then $\frac12 n (q-1)$ is odd by \cite[Proposition~2.5.10 and Proposition~2.5.13]{KL90}. Hence, $n$ is odd and $q \equiv 3 \, (4)$, so ${d=(4,q-1)=2=d'}$. Therefore, $[G:H]_2 \ge 4$ in this case as well.

Now assume that $q$ is even. In this case, we choose $H \le \mr{P \Omega}_{2n}^+(q)$ in class $\mc{C}_1$ of type $\sp{2(n-1)}{q}$. By~\cite{BHR13} for $4 \le n \le 6$ and \cite[Table 3.5.E]{KL90} for $n \ge 7$, this subgroup is always maximal and satisfies $c=1=\pi$ unless $n=4$ and the almost simple group contains the triality automorphism. So, first assume that we are not in this case. We have by \cite[Proposition~4.1.7]{KL90} that $|H|=|\sp{2(n-1)}{q}|=q^{(n-1)^2} \prod_{k=1}^{n-1} (q^{2k}-1)$. Since $q$ is even, $d=1$ and we have $[G:H]_2=q^{n-1}$, which is divisible by $4$ since $n \ge 4$. Finally, assume $n=4$ and $G$ is an almost simple group containing the triality automorphism. Then the socle contains a maximal subgroup in class $\mc{C}_1$ isomorphic to $\mr{G}_2(q)$ of index divisible by $q^6 > 4$ with $c=1=\pi$ by \cite[Table~8.50]{BHR13}. We are done with this case by Lemma \ref{l:c=1}, and the proposition is proved.
\end{proof}

To finish the finite simple groups of Lie type, we consider the exceptional groups.

\begin{proposition} 
Let $G$ be an almost simple group with socle an exceptional group $S$ of Lie type. Assume $S \not\cong \mr{G}_2(q)$ with $q=5^e$ and $e$ odd. Then $G$ has a faithful maximal subgroup of index divisible by $4$.
\end{proposition}

\begin{proof}
Let $S \nrm G \le \mr{Aut}(S)$ with $S$ a simple exceptional group of Lie type. For each simple exceptional group, we give a maximal subgroup $H \le S$ of index divisible by $4$. The maximal subgroup we choose will always satisfy $c=1=\pi$; this information is contained in the reference we give for each maximal subgroup. Then we will be done by Lemma \ref{l:c=1}. 

If $S={}^2 \mr{B}_2(q)$ with $q=2^{2n+1}$ ($n\ge 1$) and order $q^2(q-1)(q^2+1)$, then $S$ contains a maximal subgroup $H \in \mc{C}_3$ of order $(q+\sqrt{2q}+1):4$ and index $[S:H]_2=q^2/4 > 4$ by \cite[Table~8.16]{BHR13}.

If $S={}^2 {\rm G}_2(q)$ with $q=3^{2n+1}$ ($n\ge 1$) and order $q^3(q-1)(q^3+1)$, then $S$ contains a maximal subgroup $H$ of order $(q+\sqrt{3q}+1):6$ and index $[S:H]_2=4$ by \cite[Table~8.43]{BHR13}.

Let $S=\mr{G}_2(q)$ of order $q^6(q^2-1)(q^6-1)$ with $q=p^e$ and $q \neq 5^e$ with $e$ odd. If $p \ge 7$ and $q \ge 11$, then $S$ contains a maximal subgroup $H \cong \mr{PGL}_2(q)$ of index $[S:H]_2 =(q^6-1)_2 \ge 4$ by \cite[Table 8.41]{BHR13}. So there remains to check $q=2^e$ ($e>1$), $q=3^e$ ($e \ge 1$), $q=5^e$ ($e$ even) and $q=7$. If $e$ is even, then $S$ contains a maximal subgroup $H \cong \mr{G}_2(q_0)$ with $q=q_0^2$ by \cite[Tables~8.30,~8.41,~8.42]{BHR13}. Then $[S:H]_2 = (q_0^{6})_2 (q_0^2+1)_2(q_0^6+1)_2 \ge 4$ whether $p$ is even or odd. We may now assume that $e$ is odd. For $q=2^e$ and $p=q=7$, $S$ has by \cite[Tables~8.30,~8.41]{BHR13} a maximal subgroup $H \cong \mr{SL}_3(q).2$ of index $\frac12 q^3(q^3+1)$, which is divisible by $4$ in these cases. For $q=3^e$, the maximal subgroup $H \cong {}^2 {\rm G}_2(q)$ of order $q^3(q-1)(q^3+1)$ has index $[S:H]_2=(q+1)_2 (q^3-1)_2 \ge 4$ by \cite[Table 8.42]{BHR13}. 

If $S={}^3 {\rm D}_4(q)$, then $S$ has a maximal subgroup $H$ of order $4(q^4-q^2+1)$ and index $[S:H]_2 = \frac14 (q^{12})_2(q^2-1)_2(q^6-1)_2 \ge 4$ by \cite[Table 8.51]{BHR13}.

If $S={}^2 \mr{F}_4(q)$ with $q=2^{2n+1}$ ($n \ge 1$), then $|S|_2=q^{12}$ and $S$ has a maximal subgroup $H \cong \mr{SU}_3(q):2$ of index $[S:H]_2 = q^{9}/2 \ge 4$ by \cite{Ma91}.

If $S=\mr{F}_4(q)$ of order $q^{24} \prod_{k=2,6,8,12}(q^k-1)$, then $S$ has a maximal subgroup $H$ isomorphic to ${(\sl{3}{q} \circ \sl{3}{q})}.(3,q-1).2$ by \cite[Tables 7 and 8]{Cr23}. Then the index $[S:H]=\frac12 q^{18} (q^2+1)(q^3+1)^2(q^4+1)(q^6+1)$ is divisible by $4$.

If $S=\mr{E}_6(q)$, 
then $S$ has a maximal subgroup $H \cong \psl{3}{q^3}.3$ of index divisible by $q^2(q^2-1)$ by \cite[Table 9]{Cr23}. 

If $S={}^2\mr{E}_6(q)$, 
then $S$ has a maximal subgroup $H \cong \psu{3}{q^3}.3$ of index divisible by $q^2(q^2-1)$ by \cite[Table~10]{Cr23}.

If $S=\mr{E}_7(q)$, then $S$ has a maximal subgroup $H \cong \psl{2}{q^7}.7$ of index divisible by $q^2(q^2-1)$ by \cite[Table 4.1]{Cr22}. 

If $S=\mr{E}_8(q)$, then $S$ has a maximal subgroup $H \cong (\sl{2}{q} \circ \mr{E}_7(q)).(q-1,2)$ of index divisible by $q^2(q^6+1)(q^{10}+1)$ by \cite[Table~1]{LS87}.
\end{proof}

\begin{proposition} 
Let $G$ be an almost simple group with socle the exceptional group of Lie type $\mr{G}_2(q)$ with $q=5^e$ and $e \ge 1$ odd. Then every maximal subgroup of $G$ has almost odd index if and only if $e \neq 3$.
\end{proposition}

\begin{proof}
The outer automorphism group consists only of the field automorphisms, which have odd order in this case, so we need only consider the faithful maximal subgroups. We compute the index of every maximal subgroup in \cite[Table 8.41]{BHR13}. There are no novelties, so it suffices to assume that $G=\mr{G}_2(q)$ with $q=5^e$ and $e$ odd. Note that $5^e \equiv 5 \, (8)$, and so $|G|_2=(q-1)_2^2 (q+1)_2^2 = 2^6$. For all $e \ge 1$, we have the following maximal subgroups: $H \cong q^5.\mr{GL}_2(q)$ of index $[G:H]_2 = {(q+1)_2 = 2}$; $H \cong {(\sl{2}{q} \circ \sl{2}{q}).2}$ of index $[G:H]_2=1$; $H\cong \sl{3}{q}:2$ of index ${[G:H]_2} = {(q+1)_2/2}=1$; and $H \cong \su{3}{q}:2$ of index $[G:H]_2 = (q-1)_2/2 = 2$. If $e>1$, then $G$ also contains maximal subgroups $H \cong \mr{G}_2(q_0)$ with $q=q_0^r$ and $r$ prime of index $[G:H]_2 = 1$. If $q=5$, then $G$ also has maximal subgroups $H \cong 2^3 \cdot \psl{3}{2}$ of index $[G:H]_2 =1$ and $H \cong {\psu{3}{3}:2}$ of index $[G:H]_2 = 1$. If $q=5^3$, then $G$ also has maximal subgroups $H \cong \psl{2}{8}$ of index $[G:H]_2=2^3$. We have accounted for all the faithful maximal subgroups of $G$, so the proposition is proved. 
\end{proof}

Next, we consider the alternating groups.

\begin{proposition}\label{p:alt} 
Let $G$ be an almost simple group with socle a simple alternating group $\mr{A}_n$ $(n \ge 5)$. Then $G$ has a faithful maximal subgroup of index divisible by~$4$ if and only if $G$ is not isomorphic to $\mr{A}_5$, $\mr{S}_5$, $\mr{A}_6$, $\mr{S}_6$ or $\mr{A}_7$. 
\end{proposition}

\begin{proof}
For $n \le 23$, we compute in GAP and see that the proposition is true. Now let $n > 23$. First assume that $n$ is even. By \cite{LPS87}, $G$ contains maximal subgroups of the form $M=({\rm S}_m \times {\rm S}_k) \cap G$ with $n=m+k$ of index $[G:M]=\binom{n}{k}$, and $\binom{n}{3}$ is divisible by $4$. Now assume that $n$ is odd. If $n=p$ is prime, then $G$ contains a maximal subgroup of the form $M=\mr{AGL}_1(p) \cap G$ by \cite{LPS87}. So $[G:M]_2=\bigg( \frac{p!}{p(p-1)} \bigg)_2 = (p-2)!_2$, which is divisible by $4$ since $p > 23$. If $n$ is not prime, then write $n=pk$, where $p$ is the smallest prime divisor of~$n$. So $k > 3$ since $n > 23$. By \cite{LPS87}, $G$ contains a maximal subgroup of the form $M=({\rm S}_p \wr {\rm S}_k) \cap G$ of index $${[G:M]=\frac{n!}{(p!)^k k!}} = {\frac{1}{k!} \prod_{j=1}^k \binom{jp}{p}} = {\prod_{j=1}^k \frac{p}{jp} \binom{jp}{p}} = {\prod_{j=1}^{k} \binom{jp-1}{p-1}}.$$ By Kummer's theorem \cite[pp.~115--116]{Ku1852}, the $2$-adic valuation of the binomial coefficient $\binom{jp-1}{p-1}$ is equal to the number of carries when adding $(jp-1)-(p-1)=(j-1)p$ and $p-1$ in base~$2$. First assume that $p=2^l+1$ is a Fermat prime. In base~$2$, $p-1=1 \cdot 2^l$ and $3p = 1+1\cdot 2+1\cdot 2^l+1\cdot 2^{l+1}$, so adding $p-1$ and $3p$ in base~$2$ requires $2$ carries. Hence, $\binom{4p-1}{p-1}$ is divisible by $2^2$, and $[G:M]$ is divisible by $4$, as well. Finally, assume that $p$ is not a Fermat prime. In base~$2$, we have $p=\sum_{i \ge 0} a_i 2^i$, where $a_0=1$ and $a_i=1$ for at least two values of $i>0$. So when we add $p-1$ and $p$, we carry in at least two places. Hence, $\binom{2p-1}{p-1}$ is divisible by $4$, and $[G:M]$ is divisible by $4$, as well.
\end{proof}

Finally, we consider the sporadic simple groups and the Tits group ${}^2 {\rm F}_4(2)'$.

\begin{proposition}\label{p:sporadic} 
Let $G$ be an almost simple group with socle a sporadic simple group or the Tits group ${}^2{\rm F}_4(2)'$. Then $G$ has a faithful maximal subgroup of index divisible by $4$.
\end{proposition}

\begin{proof}
There are 27 simple groups to check, and 13 of them have nontrivial outer automorphism groups (all of order 2), so there are 40 cases. We check all these cases except for the monster in GAP and see that each almost simple group has a maximal subgroup of index divisible by~$4$. They are all faithful since the index of the socle is at most $2$. Finally, if $G$ is the monster sporadic simple group, then $|G|_2=2^{46}$ and $G$ has a maximal subgroup $H \cong 2.{\rm B}$, where ${\rm B}$ is the baby monster, of order $|H|_2 = 2^{42}$ by~\cite[Table 1]{DLPP24}. So $[G:H]_2 \ge 4$.
\end{proof}

Theorem \ref{thm:almost-simple-faithful} now follows by Propositions \ref{p:psln}--\ref{p:sporadic} and the classification of finite simple groups.

\section{General finite groups}\label{sec:finite-groups}
In this section, we determine the structure of finite groups whose maximal subgroups have almost odd index. We begin with a lemma.

\begin{lemma}\label{l:O2'}
Let $G$ be a finite group. Then every maximal subgroup of $G$ has almost odd index if and only if every maximal subgroup of $G/\mbf{O}_{2'}(G)$ has almost odd index.
\end{lemma}

\begin{proof}
If every maximal subgroup of $G$ has almost odd index, then the same holds for the maximal subgroups of $G/\mbf{O}_{2'}(G)$. For the converse, assume that every maximal subgroup of $G/\mbf{O}_{2'}(G)$ has almost odd index. Let $M$ be a maximal subgroup of $G$. If $\mbf{O}_{2'}(G) \le M$, then $M/\mbf{O}_{2'}(G)$ is a maximal subgroup of $G/\mbf{O}_{2'}(G)$, so $[G:M]$ is almost odd. If $\mbf{O}_{2'}(G) \not\le M$, then $G=\mbf{O}_{2'}(G)M$ and $[G:M]$ is odd. 
\end{proof}

Now we prove Theorem \ref{thm3:almost-odd}.

\begin{theorem}\label{thm:almost-odd-sect} 
Let $G$ be a nonsolvable finite group, and let $\q{G}=G/\mbf{O}_{2',2}(G)$. If the index of every maximal subgroup of $G$ is not divisible by $4$, then \mbox{$\q{G}' \cong N_1 \times \dots \times N_r$}, where each $N_i$ is a simple minimal normal subgroup of $\q{G}$ and is isomorphic to one of the simple groups
\begin{enumerate}[nolistsep,label=\textup{(\arabic*)}]
    \item $\mr{A}_6$, $\mr{A}_7$, $\psu{3}{5}$,
    \item $\psl{2}{q}$ with $q \equiv 5  \, (8)$, or
    \item $\mr{G}_2(q)$ with $q=5^{2n+1}$ $(n \ge 0$, $n \neq 1)$.
\end{enumerate}
If $N_i \cong \mr{A}_6$, then $\mr{A}_6 \le \q{G}/\Cen{\q{G}}{N_i} \le \mr{S}_6$; if $N_i \cong \mr{A}_7$, then $\q{G}/\Cen{\q{G}}{N_i} \cong \mr{A}_7$; if $N_i \cong \psu{3}{5}$, then $\psu{3}{5} \le \q{G}/\Cen{\q{G}}{N_i} \le \psu{3}{5}.2$. The Sylow $2$-subgroups of $\q{G}$ have derived length at most $2$.
\end{theorem}

\begin{proof}
In light of Lemma \ref{l:O2'}, we may assume that $\mbf{O}_{2'}(G)=1$. Let $\q{G}=G/\mbf{O}_2(G)$ and use the `bar' notation. We claim that $\mbf{O}_{2'}(\q{G})=1$. To see this, let $K/\mbf{O}_2(G) = \mbf{O}_{2'}(G/\mbf{O}_2(G))$. Then $\mbf{O}_2(G) \nrm K \nrm G$ and $K$ is solvable. By Hall's theorem \cite[Satz VI.1.7]{Hu67}, $K$ contains a Hall $2'$-subgroup, say $A$, and any two Hall $2'$-subgroups are conjugate in $K$. So $K=\mbf{O}_2(G)A$ and, by Frattini's argument, $G=K \Norm{G}{A}=\mbf{O}_2(G) \Norm{G}{A}$. If $G = \Norm{G}{A}$, then $A \le \mbf{O}_{2'}(G)=1$ and we are done. So assume that $\Norm{G}{A} \neq G$, and let $M$ be a maximal subgroup of $G$ containing $\Norm{G}{A}$. Then ${G= \mbf{O}_2(G) M}$, and so $[G:M]=2$ since the index is a power of $2$ and almost odd. Therefore,~${M \nrm G}$. Now, $\gen{A^G} \nrm G$ and $\gen{A^G} \le K$, so $A$ is a Hall $2'$-subgroup of $\gen{A^G}$. Again by Hall's theorem and the Frattini argument, $G=\gen{A^G}\Norm{G}{A}$. However, $\gen{A^G} \le M$ as $A \le M \nrm G$, so $G=\gen{A^G}\Norm{G}{A} \le M$, a contradiction.

We have shown that $\mbf{O}_{2'}(\q{G})=\mbf{O}_2(\q{G})=1$. From now on, we work in the quotient~$\q{G}=G/\mbf{O}_{2}(G)$, so in an abuse of notation we denote $\q{G}$ again by $G$ to simplify the exposition. We now have $F(G)=1$ and $F^*(G) = E(G) = N_1 \times \dots \times N_r$ is a direct product of (nonabelian) minimal normal subgroups of $G$. Let $N$ be a minimal normal subgroup of $G$. By passing to the quotient $G/\Cen{G}{N}$, we may assume that $N$ is the unique minimal normal subgroup of $G$ and $N = T_1 \times \dots \times T_k$ for some isomorphic nonabelian finite simple groups $T_i$. We will show that $k=1$ and $G$ is one of the almost simple groups in Theorem \ref{thm1:almost-simple}. 

Letting $T=T_1$, we have that $k=[G:\Norm{G}{T}]$. Choose a transversal $\{g_i\}_{i=1}^k$ for $\Norm{G}{T}$ in $G$. Let $M$ be a maximal subgroup of $\Norm{G}{T}$ containing $\Cen{G}{T}$ that does not contain $T$; that is, $M/\Cen{G}{T}$ is a faithful maximal subgroup of the almost simple group $\Norm{G}{T}/\Cen{G}{T}$. Let $H = M \cap T$ (so $1 < H < T$), and let $K=\prod_{i=1}^k H^{g_i}$. Then it follows by \cite[Theorem 4.3]{Ko86} that $\Norm{G}{K}$ is a maximal subgroup of $G$ (of so-called ``product type'') with index ${[G:\Norm{G}{K}]}=[T:H]^k=[\Norm{G}{T}/\Cen{G}{T}:M/\Cen{G}{T}]^k$. By \cite[Theorem~C]{Gu86} and its proof, if $T \not\cong \mr{A}_7$, then we can choose $M/\Cen{G}{T}$ to be a faithful maximal subgroup of $\Norm{G}{T}/\Cen{G}{T}$ of even index, so we must have $k=1$. Furthermore, if $T$ is not isomorphic to $\mr{A}_6$, $\psu{3}{5}$, $\psl{2}{q}$ with $q \equiv 5 \, (8)$ or $\mr{G}_2(q)$ with $q = 5^{2n+1}$ ($n\ge 0$, $n \neq 1$), then we can choose $M/\Cen{G}{T}$ to be a faithful maximal subgroup of $\Norm{G}{T}/\Cen{G}{T}$ with index divisible by $4$ by Theorem~\ref{thm:almost-simple-faithful}. We have shown that either $N \cong T$ is simple, where $T$ is one of the simple groups in Theorem~\ref{thm1:almost-simple}, or $N \cong T^k$ with $T \cong \mr{A}_7$. 

Now, if $T \cong \mr{A}_7$, then the almost simple group $\Norm{G}{T}/\Cen{G}{T}$ is isomorphic to~$\mr{A}_7$; otherwise, it is isomorphic to $\mr{S}_7$ and we can find a faithful maximal subgroup of index divisible by $4$, which yields a maximal subgroup of $G$ with index divisible by~$4$ using the argument of the last paragraph. Therefore, $\Norm{G}{T}=T \Cen{G}{T}$, and so $G$ is a wreath product. Assume for contradiction that $k > 1$. By \cite[Proposition~8.1]{Gu86} and the remark following, there exists a diagonal subgroup $D \le N$ such that ${G=N \Cen{G}{D}}$. Furthermore, a maximal subgroup $M$ of $G$ containing $D \Cen{G}{D}$ has index $[G:M]$ divisible by $|T|$, which is divisible by $4$. We also give a more detailed proof of this fact: By \cite[Theorem 2]{AS85}, $N$ has a complement $L$ in $G$, and $L$ normalizes a diagonal subgroup $D$ of~$N$. Note that $L$ acts transitively on the simple factors of~$N$. Since $DL$ is a proper subgroup of $G$, we may take $M$ to be a maximal subgroup of $G$ containing $DL$. Now, $[G:M]=[N:N \cap M]$ and $N \cap M$ contains $D$, so a maximal subgroup $\wt{M}$ of $N$ containing $N \cap M$ does not split as a direct product. Therefore, $\wt{M}$ is the inverse image of the diagonal subgroup of $T^2$ under a projection map from $T^k$ to $T^2$ by \cite[Corollary 1.7]{Th97}. That is, $|T|=[N:\wt{M}]$ divides $[N:N \cap M]$, and so $[G:M]$ is divisible by $4$.
 
We have proved all of Theorem \ref{thm:almost-odd-sect} except for checking that the derived length of the Sylow $2$-subgroups of the possible almost simple groups is at most $2$. We simply compute in GAP that the Sylow $2$-subgroups of $\mr{A}_6$, $\mr{S}_6$, $\mr{A}_7$, $\psu{3}{5}$ and $\psu{3}{5}.2$ have derived length $2$. A Sylow $2$-subgroup of $\mr{P \Gamma L}_2(q)$ with $q \equiv 5 \, (8)$ is contained in $\pgl{2}{q}$ since the field automorphisms have odd order, and the Sylow $2$-subgroups are dihedral of derived length $2$. Finally, the Sylow $2$-subgroups of an almost simple group with socle $\mr{G}_2(q)$ and $q=5^e$ with $e$ odd are contained in $\mr{G}_2(q)$ since the field automorphisms have odd order. Note that $|\mr{G}_2(q)|_2=2^6$ for all $q$ under consideration, so $\mr{G}_2(5) \le \mr{G}_2(q)$ with odd index, and we simply compute in GAP that the Sylow $2$-subgroups of $\mr{G}_2(5)$ have derived length $2$.
\end{proof}

\section{\texorpdfstring{$p$}{p}-nilpotency criteria}\label{sec:solvable}
In this section, we consider solvable groups, and we derive some $2$-nilpotency criteria as a consequence of our previous results.

First, we prove the $p$-nilpotency criterion of Theorem \ref{thm2:solvable}.

\begin{theorem}\label{thm:solvable-insection} 
Let $G$ be a finite $p$-solvable group, where $p$ is the smallest prime dividing the order of $G$. Every maximal subgroup of $G$ has index not divisible by $p^2$ if and only if $G$ has a normal $p$-complement.
\end{theorem}

\begin{proof}
First assume that a finite group $G$ has a normal $p$-complement $N$, and let $M$ be a maximal subgroup of~$G$. 
If $N \not\le M$, then $G=NM$ and $[G:M]=[N:M \cap  N]$ is not divisible by $p$. If $N \le M$, then $M/N$ is a maximal subgroup of the $p$-group $G/N$, so $[G:M]=[G/N:M/N]=p$. We have shown that the index of a maximal subgroup is either coprime to $p$ or equal to $p$. Note that this direction does not require $G$ to be $p$-solvable or $p$ to be the smallest prime dividing $|G|$. Indeed, this direction is a consequence of the fact that $p$-nilpotent groups are $p$-supersolvable and \cite[Satz~VI.9.2]{Hu67}.

For the converse, assume that $G$ is a $p$-solvable group, where $p$ is the smallest prime dividing $|G|$. Assume that every maximal subgroup of $G$ has index not divisible by~$p^2$. Since $G$ is $p$-solvable, $G$ has a Hall $p'$-subgroup $H$ and any two Hall $p'$-subgroups are conjugate \cite[Satz VI.1.7]{Hu67}. We are done if $H \nrm G$, so assume for contradiction that $\Norm{G}{H}<G$. Let $M$ be a maximal subgroup of $G$ containing $\Norm{G}{H}$. Since $M$ contains a Hall $p'$-subgroup of $G$, its index is a power of $p$. Combined with our assumption, this means $[G:M]=p$. As $p$ is the smallest prime dividing $|G|$, $M$ is normal in $G$. But then by Frattini's argument, ${G=M \Norm{G}{H}=M<G}$, a contradiction. Note that the proof in this direction shows that if $G$ is $p$-supersolvable for the smallest prime $p$ dividing $|G|$, then $G$ is $p$-nilpotent.
\end{proof}\textit{}


Next, we prove the $2$-nilpotency criterion of Theorem \ref{cor4:2nilpotent-1}.

\begin{theorem}\label{thm:4.2} 
Let $G$ be a finite group. If every maximal subgroup of $G$ has index not divisible by $3$ and not divisible by $4$, then $G$ has a normal $2$-complement.
\end{theorem}

\begin{proof}
Since the index of every maximal subgroup is almost odd, $G$ is either solvable, in which case we are done by Theorem \ref{thm:solvable-insection}, or $G$ is nonsolvable and has the structure described in Theorem \ref{thm3:almost-odd}. Assume for contradiction that $G$ is nonsolvable. By passing to an appropriate quotient, we can assume that $G$ is isomorphic to $\mr{A}_6$, $\mr{S}_6$, $\mr{A}_7$ or $G$ is an almost simple group with socle isomorphic to $\psl{2}{q}$ with $q \equiv 5 \, (8)$ or $\mr{G}_2(q)$ with $q=5^{2n+1}$ ($n \ge 0$, $n \neq 1$). We derive a contradiction by showing that each of these almost simple groups has a maximal subgroup of index divisible by $3$.

If $G$ is isomorphic to $\mr{A}_5 \cong \mr{PSL}_2(5)$, $\mr{S}_5$, $\mr{A}_6$, $\mr{S}_6$, $\mr{A}_7$, $\psu{3}{5}$ or $\psu{3}{5}.2$, we check by computing in GAP that $G$ has a maximal subgroup of index divisible by $3$. If $G$ is almost simple with socle isomorphic to $\psl{2}{q}$ with $q \equiv 5 \, (8)$, then $q \equiv \pm 1 \, (3)$. We may assume that $q \ge 13$, so $\psl{2}{q}$ has maximal subgroups isomorphic to $D_{q \pm 1}$ of indices $q(q \mp 1)/2$ with $c=1=\pi$ by \cite[Table~8.1]{BHR13}, and one of these indices is divisible by $3$. Finally, if $G$ is almost simple with socle $\mr{G}_2(q)$ with $q=5^{2n+1}$ ($n \ge 0$, $n \neq 1$), then $q \equiv -1 \, (3)$. The socle has maximal subgroups isomorphic to $q^5.\mr{GL}_2(q)$ of index divisible by $(q+1)_3 \ge 3$ with $c=1=\pi$ by \cite[Table~8.41]{BHR13}. The theorem is proved.
\end{proof}

To conclude the paper, we prove Theorem \ref{cor5:2nilpotent-2}, which is a recognition theorem for $p$-nilpotency.

\begin{theorem}\label{thm:4.3} 
Let $G$ be a finite group, and let $p$ be the smallest prime dividing the order of $G$. Then $G$ has a normal $p$-complement if and only if for every second maximal subgroup $H$ of $G$, the index $[G:H]$ is a power of $p$ or is not divisible by $p^2$.
\end{theorem}

\begin{proof}
First assume that a finite group $G$ has a normal $p$-complement, and let $H$ be a second maximal subgroup of~$G$. So $H$ is maximal in some maximal subgroup $M$ of~$G$. We showed in the proof of Theorem~\ref{thm:solvable-insection} that whenever a finite group has a normal $p$-complement, the index of a maximal subgroup is either equal to $p$ or coprime to $p$. As both $G$ and $M$ have normal $p$-complements, the index $[G:H]=[G:M][M:H]$ is either equal to~$p^2$ or not divisible by $p^2$. Note that this direction does not require $p$ to be the smallest prime dividing $|G|$.



For the converse, assume that the index of every second maximal subgroup of $G$ is a power of $p$ or is not divisible by $p^2$. Observe that if $M$ is a maximal subgroup of~$G$, then the index of $M$ in $G$ is a power of $p$ or is not divisible by $p^2$. Suppose there exists some maximal subgroup $M$ of $G$ with index $[G:M]=p^k>p$. Then every maximal subgroup of $M$ has $p$-power index. This means $M$ is a $p$-group, and then so is $G$. The theorem holds trivially in this case, so we may assume that $G$ is not a $p$-group and every maximal subgroup has index not divisible by $p^2$. Now, $G$ is either solvable, in which case we are done by Theorem \ref{thm:solvable-insection}, or $G$ is nonsolvable, in which case $p=2$ by the Feit--Thompson theorem and $G$ has the structure described in Theorem \ref{thm3:almost-odd}. Assume for contradiction that $G$ is nonsolvable. By passing to an appropriate quotient, we can assume that $G$ is isomorphic to $\mr{A}_6$, $\mr{S}_6$, $\mr{A}_7$, $\psu{3}{5}$, $\psu{3}{5}.2$ or $G$ is an almost simple group with socle isomorphic to $\psl{2}{q}$ with $q \equiv 5 \, (8)$ or $\mr{G}_2(q)$ with $q=5^{2n+1}$ ($n \ge 0$, $n \neq 1$). We derive a contradiction by showing that each of these almost simple groups contains a second maximal subgroup whose index is neither almost odd nor a power of $2$.

If $G$ is isomorphic to $\mr{A}_6$, $\mr{S}_6$, $\mr{A}_7$, $\psu{3}{5}$ or $\psu{3}{5}.2$, then we can compute in GAP and find a second maximal subgroup violating the hypotheses of the theorem. Suppose $G$ is almost simple with socle $S \cong \psl{2}{q}$ and $q \equiv 5 \, (8)$. The simple group $S$ has a maximal subgroup $M \cong D_{q+1}$ of index $q(q-1)/2$, which divisible by $2$ but not~$4$. Since $c=1=\pi$ for this $M$, we know that $\Norm{G}{M}$ is maximal in $G$ of the same index. Let $C$ be the characteristic cyclic subgroup of $M$ of order $(q+1)/2$. Then $\Norm{G}{M} \le \Norm{G}{C} < G$, and we must have equality by the maximality of $\Norm{G}{M}$ in~$G$. The quotient $\Norm{G}{C}/\Cen{G}{C} \le \mr{Aut}{(C)}$ is abelian and has even order since it contains an involution from the dihedral group $M$. Therefore, $\Norm{G}{M}$ has a maximal subgroup $H$ of index $2$. The subgroup $H$ is second maximal in $G$ of index $q(q-1)$, which is neither almost odd nor a power of $2$. Finally, suppose $G$ is almost simple with socle $S \cong \mr{G}_2(q)$ and $q = 5^{2n+1}$ ($n \ge 0$, $n \neq 1$). The group $S$ has a maximal subgroup $M \cong \mr{SU}_3(q):2$ of index $[S:M]=\frac12 q^3(q^3-1)$, and $M$ has a maximal subgroup $H \cong \mr{SU}_3(q)$. Then $H$ is second maximal in $S$ of index $[S:H]=q^3(q^3-1)$, which is neither almost odd nor a power of $2$. Since the outer automorphism group $\gen{\phi}$ of $S$ consists only of field automorphisms, $K=H(\gen{\phi} \cap G)$ is a second maximal subgroup $G$ with index $[G:K]=[S:H]$.
\end{proof}

\section{Computer calculations}\label{sec:gap}

In this section, we give the GAP code that we used to verify the computational claims made in the body of the paper. The following code defines the function we used to check for faithful maximal subgroups of index divisible by $4$. 

\begin{lstlisting}
# This function counts (conjugacy classes of) maximal
# subgroups of index divisible by 4. If opt=1, the
# function outputs the structure and index of each
# maximal subgroup. This allows us to check if the
# maximal subgroup is faithful.

4divisible:=function(g,opt)
    local c,i,number;
    c:=MaximalSubgroupClassReps(g);;
    number:=0;
    for i in [1..Size(c)] do;
        if Index(g,c[i]) mod 4 = 0 then 
        number := number +1;
            if opt = 1 then
            Display([StructureDescription(c[i]), 
                PrimePowersInt(Index(g,c[i]))]);
            fi;
        fi; 
    od;
    return number;
end;
\end{lstlisting}

We complete the proof of Proposition \ref{p:psln} by computing in almost simple groups with socle $\psl{3}{4}$, and we see that there always exists a faithful maximal subgroup of index divisible by $4$.

\begin{lstlisting}
G:=PSL(3,4);;n:=Size(G);
A:=AutomorphismGroup(G);;
S:=List(ConjugacyClassesSubgroups(A),Representative);;
for i in [1..Size(S)] do;
    A1:=S[i];; n1:=Size(A1);;
    if n1 >= n then
    Display([n1/n,4divisible(A1,1)]); fi; od;
\end{lstlisting}

We check almost simple groups with socle $\psu{3}{3}$, $\psu{3}{5}$ and $\psu{5}{2}$ to complete the proof of Proposition \ref{p:psu}. This calculation shows, in particular, that the maximal subgroups of $\psu{3}{5}$ and $\psu{3}{5}.2$ have almost odd index.

\begin{lstlisting}
### PSU(3,3)
G:=PSU(3,3);;n:=Size(G);
A:=AutomorphismGroup(G);;
S:=List(ConjugacyClassesSubgroups(A),Representative);;
for i in [1..Size(S)] do;
    A1:=S[i];; n1:=Size(A1);;
    if n1 >= n then
    Display([n1/n,4divisible(A1,1)]); fi; od;

### PSU(3,5)
G:=PSU(3,5);;n:=Size(G);
A:=AutomorphismGroup(G);;
S:=List(ConjugacyClassesSubgroups(A),Representative);;
for i in [1..Size(S)] do;
    A1:=S[i];; n1:=Size(A1);;
    if n1 >= n then
    Display([n1/n,4divisible(A1,1)]); fi; od;

### PSU(5,2)
G:=PSU(5,2);;n:=Size(G);
A:=AutomorphismGroup(G);;
S:=List(ConjugacyClassesSubgroups(A),Representative);;
for i in [1..Size(S)] do;
    A1:=S[i];; n1:=Size(A1);;
    if n1 >= n then
    Display([n1/n,4divisible(A1,1)]); fi; od;
\end{lstlisting}

We check almost simple groups with socle $\psl{2}{9} \cong \mr{A}_6$ to finish the proof of Proposition~\ref{p:psl2-q5}.

\begin{lstlisting}
4divisible(AlternatingGroup(6),1);                    # 0
4divisible(SymmetricGroup(6),1);                      # 0
4divisible(MathieuGroup(10),1);                       # 1
4divisible(PGL(2,9),1);                               # 1
4divisible(AutomorphismGroup(AlternatingGroup(6)),1); # 1
\end{lstlisting}

We complete the proof of Proposition~\ref{p:alt} by finding maximal subgroups of index divisible by $4$ in $\mr{A}_n$ and $\mr{S}_n$ for $5 \le n \le 23$. The case $n=6$ was addressed above. In particular, this calculation shows that the maximal subgroups of $\mr{A}_7$ have almost odd index, but $\mr{S}_7$ has a maximal subgroup of index divisible by $4$. Note that any maximal subgroup of index divisible by $4$ is faithful since the outer automorphism group of a simple alternating group is abelian, so a maximal subgroup containing the socle has prime index.

\begin{lstlisting}
for n in [5..23] do;
    Display([n, 4divisible(AlternatingGroup(n),0),
        4divisible(SymmetricGroup(n),0)]);
od;
\end{lstlisting}

We complete the proof of Proposition~\ref{p:sporadic} by computing in GAP.

\begin{lstlisting}
spornames:=["M11", "M12", "M12.2", "M22", "M22.2", "M23","M24","J1", "J2", "J2.2", "J3", "J3.2", "J4", "Co1", "Co2", "Co3", "Fi22",  "Fi22.2", "Fi23", "F3+", "F3+.2", "HS", "HS.2", "McL", "McL.2",  "He", "He.2", "Ru", "Suz", "Suz.2", "ON", "ON.2", "HN", "HN.2",  "Ly", "Th", "B", "2F4(2)'", "2F4(2)"];; # Groups to check
Display(["Number of cases to check = ",Size(spornames)]);
sportable:=List(spornames,CharacterTable);; # Load character tables
for i in [1..Size(sportable)] do;
    gtbl:=sportable[i];  
    gsize:=Size(gtbl); # Size of group i
    maxs:=Maxes(sportable[i]); 
    maxtbls:=List(maxs,CharacterTable);
    maxsize:=List(maxtbls,Size); # Sizes of maximal subgroups
    number:=0; 
    for j in [1..Size(maxsize)] do; 
        if gsize/maxsize[j] mod 4 = 0 then; 
        number:=number+1;fi; 
    od;
    Display([spornames[i],number]);
od;
\end{lstlisting}

We complete the proof of Theorem \ref{thm:almost-odd-sect} and Theorem \ref{thm:4.2} by showing that some small almost simple groups contain maximal subgroups of index divisible by $3$ and that their Sylow $2$-subgroups have derived length $2$.
\begin{lstlisting}
3divisible:=function(g)
    local c,i,number;
    c:=MaximalSubgroupClassReps(g);;
    number:=0;
    for i in [1..Size(c)] do;
        if Index(g,c[i]) mod 3 = 0 then 
        number := number +1;
    fi; od;
    return number;
end;

DerivedLength(SylowSubgroup(SimpleGroup("G2(5)"),2));

3divisible(AlternatingGroup(5)); 3divisible(SymmetricGroup(5));

3divisible(AlternatingGroup(6)); 3divisible(SymmetricGroup(6)); 
DerivedLength(SylowSubgroup(SymmetricGroup(6),2));

3divisible(AlternatingGroup(7));
DerivedLength(SylowSubgroup(AlternatingGroup(7),2));

3divisible(PSU(3,5));
c:=MaximalSubgroupClassReps(AutomorphismGroup(PSU(3,5)));
g:=c[2];;StructureDescription(g);
3divisible(g);
DerivedLength(SylowSubgroup(AutomorphismGroup(PSU(3,5)),2));
\end{lstlisting}

We complete the proof of Theorem~\ref{thm:4.3} by showing that some small almost simple groups contain second maximal subgroups whose index is neither almost odd nor a power of $2$.
\begin{lstlisting}
2max:=function(g)
    local c,c2,i,j;
    c:=MaximalSubgroupClassReps(g);;
    for i in [1..Size(c)] do;
        c2:=MaximalSubgroupClassReps(c[i]);;
        for j in [1..Size(c2)] do;
        Display([PrimePowersInt(Index(g,c2[j])),
            StructureDescription(c2[j])]);
    od;od;
end;

2max(AlternatingGroup(6)); 
2max(SymmetricGroup(6));
2max(AlternatingGroup(7));
2max(PSU(3,5));
c:=MaximalSubgroupClassReps(AutomorphismGroup(PSU(3,5)));
g:=c[2];;StructureDescription(g);
2max(g);
\end{lstlisting}



\bigskip \footnotesize

\textsc{Department of Mathematics and Statistics, Binghamton University,
    Binghamton, NY 13902-6000, USA}\par\nopagebreak
  \textit{E-mail address}: \texttt{cschroe2@binghamton.edu}

\vspace{\baselineskip}

\textsc{Department of Mathematics and Statistics, Binghamton University,
    Binghamton, NY 13902-6000, USA}\par\nopagebreak
  \textit{E-mail address}: \texttt{htongvie@binghamton.edu}
  
\end{document}